\documentclass[11pt,leqno]{amsart}

\usepackage{amssymb}

\usepackage{times}

\usepackage{mathrsfs}

\usepackage{verbatim}

\input{colordvi}

\theoremstyle{plain}

\newtheorem{theorem}{Theorem}[section]

\theoremstyle{remark}

\newtheorem{remark}[theorem]{Remark}

\theoremstyle{plain}

\newtheorem{corollary}[theorem]{Corollary}

\newtheorem{lemma}[theorem]{Lemma}

\newtheorem{proposition}[theorem]{Proposition}

\newtheorem{definition}[theorem]{Definition}

\numberwithin{equation}{section}

\newcommand\sgn{\operatorname{sgn}}

\newcommand\diver{\operatorname{div}}

\newcommand\curl{\operatorname{curl}}

\newcommand{\eps}{\varepsilon}

\newcommand{\al}{\alpha}

\def\Nat{{\mathbb N}}

\def\Rnu{{\mathbb R}}

\def\Rd{{\Rnu^d~}}

\def\liml{\lim\limits}

\def\liminfl{\liminf\limits}

\def\limsupl{\limsup\limits}

\def\suml{\sum\limits}

\def\supl{\sup\limits}

\def\maxl{\max\limits}

\def\intl{\int\limits}

\def\supl{\mathop{\sup}\limits}

\begin{document}
\title{Beale-Kato-Majda type condition for Burgers equation}

\author{Ben Goldys, Misha Neklyudov}

\address{School of Mathematics and Statistics,
University of NSW, Sydney,  Australia} \thanks{This work was
supported by an ARC Discovery grant}
%\date{\today}

\maketitle
\begin{abstract}
We consider a multidimensional Burgers equation on the torus
$\mathbb{T}^d$ and the whole space $\Rd$. We show that, in case of
the torus, there exists a unique global solution in Lebesgue
spaces. For a torus we also provide estimates on the large time
behaviour of solutions. In the case of $\Rd$ we establish the
existence of a unique global solution if a Beale-Kato-Majda type
condition is satisfied. To prove these results we use the
probabilistic arguments which seem to be new.
\end{abstract}
In this paper we are concerned with the following multidimensional
Burgers equation:
\begin{eqnarray}
\frac{\partial u^i}{\partial t}+\suml_{j=1}^nu^j\frac{\partial
u^i}{\partial x_j}=\nu\triangle u^i+f^i,\quad,t\in[0,T]\label{eqn:BurgersEquation-1}\\
u(0)=u_0,i=1,\ldots,d, x\in \mathcal{O},\nonumber\\
u_0\in L^p(\mathcal{O},\Rd),f\in
L^p(0,T;L^p(\mathcal{O},\Rd)),p\geq d,\nonumber
\end{eqnarray}
where $\mathcal{O}$ is either the torus $\mathbb{T}^d$ or the full
space $\mathbb{R}^d$. Equations of this type arise in the theory
of conservation laws, see for example \cite{serre} and are also
known as simplified models of turbulence.
\par\noindent
If the external force $f$ is of potential type, $f=\nabla U$ and
the initial condition $u_0=\nabla U_0$ is of gradient type as
well, the existence and uniqueness of solutions is well known, see
for example \cite{Taylor} and references therein. These
assumptions however, are too restrictive in many problems. For
example the Burgers equation with data of non-potential type
arises in some problems of gas dynamics and inelastic granular
media (see \cite{chen}). It is also important to consider a more
general Burgers equation in the analysis of turbulence. The
question of the existence and uniqueness of solutions in case of
data $f,u_0$ of non-gradient type seems to be completely open. In
this paper we will consider a general case, where $f$ and $u_0$
need not be of gradient type.  Our main result is that under some,
rather mild conditions, the existence of a unique global solution
in the whole space is implied by a version of the Beale-Kato-Majda
condition, that is well known in the theory of the Navier-Stokes
equation. %The existence of a unique global solution in the
%gradient case is an immediate corollary of this result.
Also we prove, without any additional assumptions, the existence
and uniqueness of global solution of Burgers equation on the
torus. In the last part of this paper we obtain an upper bound for
the growth of solutions for time tending to infinity.
\par\noindent
Let us recall some standard notations that will be used throughout
the paper. Suppose that $H^{\al,p}(\mathcal{O})$ - closure of
$C_0^{\infty}(\mathcal{O})$ w.r. to the norm
$||f||_{\al,p}=||(I-\triangle)^{\frac{\al}{2}}f||_p$, $\al\in
\mathbb{R}$, $p\geq 1$. In what follows we use the notation
$F(u,v)=(u\nabla)v$, $F(u)=F(u,u)$,
$\cdot'=\frac{\partial}{\partial t}$.
\begin{definition}\label{def:MildSolBurgersEqn}
We say that $u\in L^{\infty}(0,T;L^p(\mathcal{O},\Rd))$ is a mild
solution of Burgers equation with the initial condition $u_0\in
L^p(\mathcal{O},\Rd)$ and force $f\in
L^1(0,T;L^p(\mathcal{O},\Rd))$ if $F(u)\in
L^1(0,T;L^p(\mathcal{O},\Rd))$ and $u$ satisfies following
equality
\begin{equation}
u(t)=S_{t}^{\nu}u_0-\intl_0^tS_{t-s}^{\nu}(F(u(s))-f(s))ds,t\in
[0,T].\label{eqn:BurgersEquation-2}
\end{equation}
where $\{S_t^{\nu}=e^{\nu t\triangle}\}_{t\geq
0}:\mathcal{O}\to\Rd$ is a heat semigroup on $\mathcal{O}$. We
assume that $S_t^{\nu}$ acts on vector functions componentwise.
\end{definition}
\section{Local existence of solution}
The local existence of solution to Burgers equation in
$L^p(\mathcal{O},\Rd)$ spaces can be shown in the same way as for
the Navier-Stokes equation (see
\cite{FujitaKato-1964},\cite{FujitaKato-1962},\cite{Kato-1984},\cite{Weissler-1980},\cite{FabesJonesRiviere-1972}
and others). Here we only state main points of the proof following
the work of Weissler \cite{Weissler-1980}.

We will use following abstract theorem proved in
\cite{Weissler-1980}(p.222, theorem 2), see also
\cite{FujitaKato-1962} and \cite{Kato-1984}.
\begin{theorem}\label{thm:AbstractExistence}
Let $W$, $X$, $Y$, $Z$ be Banach spaces continuously embedded in
some topological vector space $\mathcal{X}$. $R_t=e^{tA},t\geq 0$
be $C_0$-semigroup on X, which satisfies the following additional
conditions
\begin{trivlist}
\item[(a1)] For each $t>0$, $R_t$ extends to a bounded map $W\to
X$. For some $a>0$ there are positive constants $C$ and $T$ such
that
\begin{equation}
|R_th|_X\leq Ct^{-a}|h|_W,h\in W,t\in(0,T].\label{eqn:a1Cond}
\end{equation}
\item[(a2)]For each $t>0$, $R_t$ is a bounded map $X\to Y$. For
some $b>0$ there are positive constants $C$ and $T$ such that
\begin{equation}
|R_th|_Y\leq Ct^{-b}|h|_X,h\in X,t\in(0,T].\label{eqn:a2Cond}
\end{equation}
Furthermore, function $|R_th|_Y\in C((0,T]), h\in X$ and
\begin{equation}
\liml_{t\to 0+}t^b|R_th|_Y=0,\forall h\in X.\label{eqn:a2Cond'}
\end{equation}

\item[(a3)]For each $t>0$, $R_t$ is a bounded map $X\to Z$. For
some $c>0$ there are positive constants $C$ and $T$ such that
\begin{equation}
|R_th|_Z\leq Ct^{-c}|h|_X,h\in X,t\in(0,T].\label{eqn:a3Cond}
\end{equation}
Furthermore, function $|R_th|_Z\in C((0,T]), h\in X$ and
\begin{equation}
\liml_{t\to 0+}t^c|R_th|_Z=0,\forall h\in X.\label{eqn:a3Cond'}
\end{equation}
\end{trivlist}
Let also $G:Y\times Z\to W$ be a bounded bilinear map, and let
$G(u)=G(u,u),u\in Y\cap Z$, $f\in L^{\infty}(0,T;W)$. Assume also
that $a+b+c\leq 1$.

Then for each $u_0\in X$ there is $T>0$ and unique function
$u:[0,T]\to X$ such that:
\begin{trivlist}
\item[(a)] $u\in C([0,T],X)$, $u(0)=u_0$. \item[(b)] $u\in
C((0,T],Y)$, $\liml_{t\to 0+}t^b|u(t)|_Y=0$. \item[(c)] $u\in
C((0,T],Z)$, $\liml_{t\to 0+}t^c|u(t)|_Z=0$.\item[(d)]
$$
u(t)=R_tu_0+\intl_0^tR_{t-\tau}(G(u(\tau))+f(\tau))d\tau,t\in
[0,T]
$$
\end{trivlist}
\end{theorem}
\begin{remark}
Weissler \cite{Weissler-1980} considers only the case of $f=0$.
The general case follows similarly (see appendix for the proof).
\end{remark}
In the next proposition we will summarize properties of heat
semigroup $S_t^{\nu}=e^{\nu t\triangle},t\geq 0$ on $\mathcal{O}$.
\begin{proposition}\label{prop:HeatSemigroup}
\begin{trivlist}
\item[(i)]
\begin{eqnarray}
|\nabla^me^{t\triangle}h|_{L^q(\mathcal{O},\Rd)}\leq c
t^{-\frac{m}{2}-\frac{d}{2r}}|h|_{L^p(\mathcal{O},\Rd)},t\in (0,T],\label{eqn:HeatSemigroupEstimate-1}\\
\frac{1}{r}=\frac{1}{p}-\frac{1}{q},1<p\leq q<\infty,h\in
L^p(\mathcal{O},\Rd).\nonumber
\end{eqnarray}
Furthermore,
\begin{equation}
\liml_{t\to
0+}t^{\frac{m}{2}+\frac{d}{2r}}|\nabla^me^{t\triangle}h|_{L^q(\mathcal{O},\Rd)}=0,
\quad h\in L^p(\mathcal{O},\Rd).\label{eqn:HeatSemigroupLimit-1}
\end{equation}
\item[(ii)]Let $p\in (1,\infty)$ and $\al<\beta$. Then for any
$t>0$ $e^{t\triangle}$ is a bounded map
$H^{\al,p}(\mathcal{O},\Rd)\to H^{\beta,p}(\mathcal{O},\Rd)$.
Moreover, for each $T>0$ there exists $C=C(p,\al,\beta)$, such
that
\begin{equation}
|e^{t\triangle}h|_{H^{\beta,p}(\mathcal{O},\Rd)}\leq
Ct^{(\al-\beta)/2}|h|_{H^{\al,p}(\mathcal{O},\Rd)},t\in (0,T],h\in
H^{\al,p}(\mathcal{O},\Rd). \label{eqn:HeatSemigroupEstimate-2}
\end{equation}
Furthermore,
\begin{equation}
\liml_{t\to
0+}t^{(\beta-\al)/2}|e^{t\triangle}h|_{H^{\beta,p}}=0,\quad h\in
H^{\al,p}(\mathcal{O},\Rd).\label{eqn:HeatSemigroupLimit-2}
\end{equation}
\item[(iii)] Let $p\in(1,\infty)$. Then for any $t>0$,
$e^{t\triangle}:L^p(\mathcal{O},\Rd)\to H^{1,p}(\mathcal{O},\Rd)$
is a bounded map. Moreover, for each $T>0$ there exists
$C=C(p,T)$, such that
\begin{equation}
|e^{t\triangle}h|_{H^{1,p}(\mathcal{O},\Rd)}\leq
Ct^{-\frac{1}{2}}|h|_{L^p(\mathcal{O},\Rd)},t\in (0,T],h\in
L^p(\mathcal{O},\Rd).\label{eqn:HeatSemigroupEstimate-3}
\end{equation}
Furthermore,
\begin{equation}
\liml_{t\to
0+}t^{\frac{1}{2}}|e^{t\triangle}h|_{H^{1,p}(\mathcal{O},\Rd)}=0,\quad
h\in L^{p}(\mathcal{O},\Rd).\label{eqn:HeatSemigroupLimit-3}
\end{equation}
\end{trivlist}
\end{proposition}
\begin{proof}
The results above are well known in case of $\mathcal O=\mathbb
R^d$. If $\mathcal O=\mathbb T^d$ then the lemma is well known for
the Dirichlet boundary conditions, see for example books by
Lunardi: Analytic semigroups and optimal regularity in parabolic
problems or by Souplet: Superlinear parabolic problems. Analogous
statements for the periodic Laplacian follow easily by the same
method.
\end{proof}
Now we can formulate the theorems:

\begin{theorem}\label{thm:LocalExistence-2}
For all $u_0\in L^p(\mathcal{O},\Rd)$, $f\in
L^{\infty}([0,T],L^{\frac{2p}{3}}(\mathcal{O},\Rd)),p\geq d$ there
exists
$T_0=T_0(\nu,|u_0|_{L^p(\mathcal{O},\Rd)},|f|_{L^{\frac{2p}{3}}(\mathcal{O},\Rd)})>0$
such that there exists unique mild solution $u\in
L^{\infty}(0,T_0;L^p(\mathcal{O},\Rd))$ of Burgers equation.
Furthermore
\begin{trivlist}
\item[(a)]$u:[0,T_0]\to L^p(\mathcal{O},\Rd)$ is continuous and
$u(0)=u_0$. \item[(b)]$u:(0,T_0]\to L^{2p}(\mathcal{O},\Rd)$ is
continuous and $\liml_{t\to
0}t^{\frac{d}{4p}}|u(t)|_{L^{2p}(\mathcal{O},\Rd)}=0$.
\item[(c)]$u:(0,T_0]\to H^{1,p}(\mathcal{O},\Rd)$ is continuous
and $\liml_{t\to
0}t^{\frac{1}{2}}|u(t)|_{H^{1,p}(\mathcal{O},\Rd)}=0$.
\end{trivlist}
\end{theorem}
\begin{proof}[Proof of Theorem \ref{thm:LocalExistence-2}]
We use theorem \eqref{thm:AbstractExistence} with following data
$X=L^p(\mathcal{O},\Rd)$, $Y=L^{2p}(\mathcal{O},\Rd)$,
$Z=H^{1,p}(\mathcal{O},\Rd)$,
$W=L^{\frac{2p}{3}}(\mathcal{O},\Rd)$. Then it follows from
H\"{o}lder inequality that $F:L^{2p}(\mathcal{O},\Rd)\times
H^{1,p}(\mathcal{O},\Rd)\to L^{\frac{2p}{3}}(\mathcal{O},\Rd)$ is
a bounded bilinear map. Conditions \eqref{eqn:a1Cond} is satisfied
with $a=\frac{d}{4p}$ by estimate
\eqref{eqn:HeatSemigroupEstimate-1}. Conditions
\eqref{eqn:a2Cond},\eqref{eqn:a2Cond'} are satisfied with
$b=\frac{d}{4p}$ by \eqref{eqn:HeatSemigroupEstimate-1} and
\eqref{eqn:HeatSemigroupLimit-1}. Conditions
\eqref{eqn:a3Cond},\eqref{eqn:a3Cond'} are satisfied with
$c=\frac{1}{2}$ by \eqref{eqn:HeatSemigroupEstimate-3} and
\eqref{eqn:HeatSemigroupLimit-3}.
\end{proof}
\begin{corollary}\label{cor:ClassicalSolutionReg}
Let $p\geq d$, $\theta\in (0,1)$, $u_0\in L^p(\mathcal{O},\Rd)$,
$f\in L^{\infty}([0,T],L^{\frac{2p}{3}}(\mathcal{O},\Rd)\cap
L^p(\mathcal{O},\Rd))$, $f\in
C^{\theta}([\eps,T],L^p(\mathcal{O},\Rd))$, $\forall \eps>0$.
%($C^{\theta}$--H\"{o}lder continuous functions of order $\theta$)
Then there exist $T_2>0$ such that $u\in
C^1((0,T_2];L^p(\mathcal{O},\Rd))\cap
C((0,T_2];H^{2,p}(\mathcal{O},\Rd))\cap
C^{\theta}([\eps,T_2],H^{2,p}(\mathcal{O},\Rd))\cap
C^{1+\theta}([\eps,T_2],L^p(\mathcal{O},\Rd))$, $\forall \eps>0$
and $u$ satisfies Burgers equation
\begin{equation}
u'=\nu\triangle u-F(u(t))+f(t),\label{eqn:BurgersEquation-11}
\end{equation}
\end{corollary}
\begin{proof}
By theorem \ref{thm:LocalExistence-2} we have that $u(t)\in
L^{2p}(\mathcal{O},\Rd),t\in (0,T_0]$. %, u(t)\in H^{1,2p}(\mathcal{O},\Rd),t>0$.
%%%% correction
Let us show that there exist $T_1$ such that $u\in
C((0,T_1],H^{1,2p}(\mathcal{O},\Rd))$ and $\liml_{t\to
0}t^{\frac{1}{2}}|u(t)|_{H^{1,2p}(\mathcal{O},\Rd)}=0$. We apply
Theorem \ref{thm:AbstractExistence} with following data
$X=Y=L^p(\mathcal{O},\Rd)$, $Z=H^{1,2p}(\mathcal{O},\Rd)$,
$W=L^{\frac{2p}{3}}(\mathcal{O},\Rd)$. Then it follows from
H\"{o}lder inequality that $F:L^p(\mathcal{O},\Rd)\times
H^{1,2p}(\mathcal{O},\Rd)\to L^{\frac{2p}{3}}(\mathcal{O},\Rd)$ is
a bounded bilinear map. Conditions \eqref{eqn:a1Cond} is satisfied
with $a=\frac{d}{4p}$ by estimate
\eqref{eqn:HeatSemigroupEstimate-1}. Conditions
\eqref{eqn:a2Cond},\eqref{eqn:a2Cond'} are satisfied with
arbitrary $b>0$ because heat semigroup is analytic on
$L^p(\mathcal{O},\Rd)$. Conditions
\eqref{eqn:a3Cond},\eqref{eqn:a3Cond'} are satisfied with
$c=\frac{1}{2}$ by \eqref{eqn:HeatSemigroupEstimate-3} and
\eqref{eqn:HeatSemigroupLimit-3}.

As the result by part c of the Theorem \ref{thm:AbstractExistence}
we get existence of $T_1$ such that $u\in
C((0,T_1],H^{1,2p}(\mathcal{O},\Rd))$ and $\liml_{t\to
0}t^{\frac{1}{2}}|u(t)|_{H^{1,2p}(\mathcal{O},\Rd)}=0$. Put
$T_2=\min\{T,T_0,T_1\}$.
%%%%

Therefore, we have
\begin{eqnarray}
|F(u)|_{L^1(0,T_2;L^p(\mathcal{O},\Rd))}\leq\intl_0^{T_2}|u(s)|_{L^{2p}(\mathcal{O},\Rd)}|\nabla
u|_{L^{2p}(\mathcal{O},\Rd)}ds\nonumber\\
\leq\intl_0^{T_2}\frac{1}{s^{\frac{d}{4p}+\frac{1}{2}}}\supl_s(s^{\frac{d}{4p}}|u(s)|_{L^{2p}(\mathcal{O},\Rd)})\supl_s(s^{\frac{1}{2}}|u(s)|_{H^{1,2p}(\mathcal{O},\Rd)})ds\nonumber\\
\leq\supl_s(s^{\frac{d}{4p}}|u(s)|_{L^{2p}(\mathcal{O},\Rd)})\supl_s(s^{\frac{1}{2}}|u(s)|_{H^{1,2p}(\mathcal{O},\Rd)}){T_2}^{\frac{1}{2}-\frac{d}{4p}}<\infty\label{eqn:aux-a}
\end{eqnarray}
Let us show that $F(u(\cdot)):[\eps,T_2]\to L^p(\mathcal{O},\Rd)$
is a H\"{o}lder continuous for any $\eps>0$. Then the result will
follow from theorem 4.3.4, p.137 in \cite{Lunardi},
\eqref{eqn:aux-a} and assumption $f\in
L^1([0,T];L^p(\mathcal{O},\Rd))\cap
C^{\theta}([\eps,T],L^p(\mathcal{O},\Rd)),$ $\forall \eps>0$.
Since $F:H^{1,2p}(\mathcal{O},\Rd)\to L^p(\mathcal{O},\Rd)$ is
locally Lipschitz it is easy to notice that it is enough to prove
that $u:[\eps,T_2]\to H^{1,2p}(\mathcal{O},\Rd)$ is a H\"{o}lder
continuous for any $\eps>0$. Since we have representation
\begin{equation}
u(t)=S_{t-\eps}^{\nu}u(\eps)-\intl_{\eps}^tS_{t-s}^{\nu}(F(u(s))-f(s))ds,t\in
[\eps,T_2].\label{eqn:BurgersEquation-3}
\end{equation}
for $u$ it is enough to show that each term of this representation
is H\"{o}lder continuous. Similarly to \eqref{eqn:aux-a} we have
\begin{equation}
\supl_{t\in[0,T_2]}t^{\frac{1}{2}+\frac{d}{4p}}|F(u(t))|_{L^p(\mathcal{O},\Rd)}\leq
\supl_ss^{\frac{d}{4p}}|u(s)|_{L^{2p}(\mathcal{O},\Rd)}\supl_ss^{\frac{1}{2}}|u(s)|_{H^{1,2p}(\mathcal{O},\Rd)}<\infty
\end{equation}
and it follows by proposition 4.2.3 part (i), p.130 of
\cite{Lunardi} that $\intl_0^tS_{t-s}^{\nu}F(u(s))ds\in
C^{\frac{1}{2}-\frac{d}{4p}}(0,T_2;L^p(\mathcal{O},\Rd))$.
Similarly, we have that $\intl_{\eps}^tS_{t-s}^{\nu}f(s)ds\in
C^{\theta}(0,T_2;L^p(\mathcal{O},\Rd))$ and the result follows.
\end{proof}
\begin{corollary}\label{cor:ClassicalSolutionReg-2}
Suppose that assumptions of the corollary
\eqref{cor:ClassicalSolutionReg} are satisfied. Assume also that
$f\in C^{\theta}([\eps,T],H^{k,p}(\mathcal{O},\Rd))$, $\forall
\eps>0$ for some $k\in \Nat$. Then $u\in
C^{\theta}([\eps,T],H^{k+2,p}(\mathcal{O},\Rd))\cap
C^{1+\theta}([\eps,T],H^{k,p}(\mathcal{O},\Rd))$, $\forall
\eps>0$.
\end{corollary}
\begin{proof}
We will show the result for $k=1$. General case follows similarly.
We have that $u(t)\in L^{2p}(\mathcal{O},\Rd),t>0$. As a result,
following the proof of the previous corollary we can get that
\begin{equation}
u\in C^{\theta}([\eps,T],H^{2,2p}(\mathcal{O},\Rd))\cap
C^{1+\theta}([\eps,T],L^{2p}(\mathcal{O},\Rd)),\forall
\eps>0.\label{eqn:ClassicalSolutionReg-3}
\end{equation}
Therefore, we have following estimates for nonlinearity
\begin{eqnarray}
|F(u)|_{C^{\theta}([\eps,T],L^p(\mathcal{O},\Rd))}\leq
|u|_{L^{\infty}(\eps,T;L^{2p}(\mathcal{O},\Rd))}|\nabla
u|_{C^{\theta}([\eps,T],L^{2p}(\mathcal{O},\Rd))}\nonumber\\
+|\nabla
u|_{L^{\infty}(\eps,T;L^{2p}(\mathcal{O},\Rd))}|u|_{C^{\theta}([\eps,T],L^{2p}(\mathcal{O},\Rd))}<\infty\label{eqn:ClassicalSolutionReg-4}
\end{eqnarray}
where we have used \eqref{eqn:ClassicalSolutionReg-3}.
Furthermore,
\begin{eqnarray}
|\nabla F(u)|_{C^{\theta}([\eps,T],L^p(\mathcal{O},\Rd))}&\leq& C
|\nabla u|_{L^{\infty}(\eps,T;L^{2p}(\mathcal{O},\Rd))}|\nabla
u|_{C^{\theta}([\eps,T],L^{2p}(\mathcal{O},\Rd))}\nonumber\\
&+&|u|_{L^{\infty}(\eps,T;L^{2p}(\mathcal{O},\Rd))}|\triangle
u|_{C^{\theta}([\eps,T],L^{2p}(\mathcal{O},\Rd))}\label{eqn:ClassicalSolutionReg-5}\\
&+&|\triangle
u|_{L^{\infty}(\eps,T;L^{2p}(\mathcal{O},\Rd))}|u|_{C^{\theta}([\eps,T],L^{2p}(\mathcal{O},\Rd))}<\infty,\nonumber
\end{eqnarray}
where we have used \eqref{eqn:ClassicalSolutionReg-3}. Thus,
combining \eqref{eqn:ClassicalSolutionReg-4} and
\eqref{eqn:ClassicalSolutionReg-5} we get $F(u)\in
C^{\theta}([\eps,T],H^{1,p})$, $\forall \eps>0$. In the same time,
by assumption we have that $f\in
C^{\theta}([\eps,T],H^{1,p}(\mathcal{O},\Rd))$, $\forall \eps>0$.
Therefore by maximal regularity result, theorem 4.3.1, p.134 of
\cite{Lunardi}, it follows that $u\in
C^{\theta}([\eps,T],H^{3,p}(\mathcal{O},\Rd))\cap
C^{1+\theta}([\eps,T],H^{1,p}(\mathcal{O},\Rd))$.
\end{proof}
In the next lemma we will show that either local solution defined
in previous theorems is global or it blows up. Let us denote
$T_{max}$ maximal existence time of solution.
\begin{lemma}\label{lem:LocalBlowUpBehaviour}
Assume that $u_0\in L^p(\mathcal{O},\Rd)$, $f\in
L^{\infty}([0,T],L^{\frac{2p}{3}}(\mathcal{O},\Rd)\cap
L^p(\mathcal{O},\Rd))$, $p>d$ and $T_{max}<\infty$. Let $u\in
L^{\infty}([0,T_{max});L^p(\mathcal{O},\Rd))$ be maximal local
mild solution of Burgers equation \eqref{eqn:BurgersEquation-2}.
Then
\begin{equation}
\limsupl_{t\nearrow
T_{max}}|u(t)|_{L^p(\mathcal{O},\Rd)}^2=\infty.\label{eqn:LocalBlowUpBehaviour-1}
\end{equation}
\end{lemma}
\begin{proof}[Proof of Lemma \ref{lem:LocalBlowUpBehaviour}]
We will argue by contradiction. Assume that
\begin{equation}
\limsupl_{t\nearrow
T_{max}}|u(t)|_{L^p(\mathcal{O},\Rd)}^2<\infty.\label{eqn:LocalBlowUpBehaviour-2}
\end{equation}
Then there exist $T_1$ such that
\begin{equation}
K_1=\supl_{t\in
[T_1,T_{max})}|u(t)|_{L^p(\mathcal{O},\Rd)}<\infty.\label{eqn:LocalLpBound}
\end{equation}
We will show that there exist $C,\al>0$ such that
\begin{equation}
|u(t)-u(\tau)|_{L^p(\mathcal{O},\Rd)}\leq
C|t-\tau|^{\al},t,\tau\in [T_2,T_{max}),T_1\leq
T_2<T_{max}.\label{eqn:UniformContinuity-1}
\end{equation}
Then it follows from \eqref{eqn:LocalBlowUpBehaviour-2} and
\eqref{eqn:UniformContinuity-1} %and Arzela-Ascoli Theorem
that there exist $y\in L^p$ such that
\begin{equation}
\liml_{t\nearrow T_{max}}
|u(t)-y|_{L^p(\mathcal{O},\Rd)}=0,\label{eqn:LocalBlowUpBehaviour-3}
\end{equation}
and we have a contradiction with definition of $T_{max}$. Thus, we
need to show \eqref{eqn:UniformContinuity-1}. Let us show first
that there exist $T_3<T_{max}$ such that
\begin{equation}
K_2=\supl_{t\in
[T_3,T_{max})}|u(t)|_{H^{1,p}(\mathcal{O},\Rd)}<\infty.\label{eqn:LocalH1pBound}
\end{equation}
It is enough to show
\begin{equation}
\supl_{t\in [T_3,T_{max})}|\nabla
u(t)|_{L^{p}(\mathcal{O},\Rd)}<\infty,\label{eqn:LocalH1pBound-2}
\end{equation}
for some $T_1\leq T_3<T_{max}$. Indeed, \eqref{eqn:LocalH1pBound}
immediately follows from \eqref{eqn:LocalLpBound} and
\eqref{eqn:LocalH1pBound-2}. We have
\begin{equation}
\nabla u(t)=\nabla S_{t}^{\nu}u_0-\intl_0^t\nabla
S_{t-s}^{\nu}(F(u(s))-f(s))ds.
\end{equation}
Hence,
\begin{eqnarray}
|\nabla u(t)|_{L^{p}(\mathcal{O},\Rd)}&\leq& |\nabla
S_{t}^{\nu}u_0|_{L^p(\mathcal{O},\Rd)}\nonumber\\
&+&\intl_0^t|\nabla
S_{t-s}^{\nu}f(s)|_{L^p(\mathcal{O},\Rd)}ds+\intl_0^t|\nabla
S_{t-s}^{\nu}F(u(s))|_{L^p(\mathcal{O},\Rd)}ds\nonumber\\
&\leq&
\frac{C|u_0|_{L^p(\mathcal{O},\Rd)}}{t^{1/2}}+\intl_0^t\frac{|f(s)|_{L^p(\mathcal{O},\Rd)}}{|t-s|^{1/2}}ds\nonumber\\
&+&C\intl_0^t\frac{|S_{(t-s)/2}^{\nu}F(u(s))|_{L^p(\mathcal{O},\Rd)}}{|t-s|^{1/2}}ds\nonumber\\
&\leq&
\frac{C|u_0|_{L^p(\mathcal{O},\Rd)}}{t^{1/2}}+2\sqrt{t}\supl_{s\in[0,t]}|f(s)|_{L^p(\mathcal{O},\Rd)}\nonumber\\
&+& C\intl_0^t\frac{|F(u(s))|_{L^{p/2}(\mathcal{O},\Rd)}}{|t-s|^{1/2+d/(2p)}}ds\nonumber\\
&\leq&
\frac{C|u_0|_{L^p(\mathcal{O},\Rd)}}{t^{1/2}}+2\sqrt{t}\supl_{s\in[0,t]}|f(s)|_{L^p(\mathcal{O},\Rd)}\nonumber\\
&+& C\intl_0^t
\frac{|u(t)|_{L^p(\mathcal{O},\Rd)}}{|t-s|^{1/2+d/(2p)}}|\nabla
u(t)|_{L^p(\mathcal{O},\Rd)}ds\nonumber\\
&\leq &
\frac{C|u_0|_{L^p(\mathcal{O},\Rd)}}{t^{1/2}}+2\sqrt{t}\supl_{s\in[0,t]}|f(s)|_{L^p(\mathcal{O},\Rd)}\nonumber\\
&+& CK\intl_0^t\frac{|\nabla
u(t)|_{L^p(\mathcal{O},\Rd)}}{|t-s|^{1/2+d/(2p)}}ds,
\end{eqnarray}
where second and third inequalities follow from
\eqref{eqn:HeatSemigroupEstimate-1}, forth inequality follows from
H\"older inequality and assumption \eqref{eqn:LocalLpBound} is
used in the fifth one. Now if $\frac{1}{2}+\frac{d}{2p}<1$ (i.e.
if $p>d$) we can use Gronwall inequality (\cite{Henry-1981}, Lemma
7.1.1, p. 188) to conclude that the estimate
\eqref{eqn:LocalH1pBound-2} holds. Thus we get an estimate
\eqref{eqn:LocalH1pBound}.

Now we can turn to the proof of \eqref{eqn:UniformContinuity-1}.
We have
\begin{equation}
u(t)-u(\tau)=S_{t-\tau}^{\nu}u(\tau)-u(\tau)+\intl_{\tau}^tS_{t-s}^{\nu}(f(s)-F(u(s)))ds.
\end{equation}
Then
\begin{eqnarray}
|u(t)-u(\tau)|_{L^p(\mathcal{O},\Rd)}\leq
|S_{t-\tau}^{\nu}u(\tau)-u(\tau)|_{L^p(\mathcal{O},\Rd)}+|\intl_{\tau}^tS_{t-s}^{\nu}f(s)ds|_{L^p(\mathcal{O},\Rd)}\nonumber\\
+|\intl_{\tau}^tS_{t-s}^{\nu}F(u(s))ds|_{L^p(\mathcal{O},\Rd)}=(I)+(II)+(III).
\end{eqnarray}
First term can be estimated as follows
\begin{eqnarray}
(I)=|\intl_{\tau}^t\nu\triangle
S_s^{\nu}u(s)ds|_{L^p(\mathcal{O},\Rd)}\leq\nu\intl_{\tau}^t|\nabla
S_s^{\nu}(\nabla
u(s))|_{L^p(\mathcal{O},\Rd)}ds\nonumber\\
\leq\nu\intl_{\tau}^t\frac{|\nabla
u(s)|_{L^p(\mathcal{O},\Rd)}}{s^{1/2}}ds\leq
K_2t^{1/2}|t-\tau|.\label{eqn:term1Est}
\end{eqnarray}
For the second term we have
\begin{equation}
(II)\leq
\supl_{s\in[\tau,t]}|f(s)|_{L^p(\mathcal{O},\Rd)}|t-\tau|.\label{eqn:term2Est}
\end{equation}
Third term is estimated as follows
\begin{eqnarray}
(III)\leq\intl_{\tau}^t\frac{|F(u(s))|_{L^{p/2}(\mathcal{O},\Rd)}}{|t-s|^{\frac{d}{2p}}}ds\leq
\intl_{\tau}^t\frac{|u(t)|_{L^p(\mathcal{O},\Rd)}|\nabla
u(t)|_{L^p(\mathcal{O},\Rd)}}{|t-s|^{\frac{d}{2p}}}ds\nonumber\\
\leq CK_2^2|t-\tau|^{1-\frac{d}{2p}},\label{eqn:term3Est}
\end{eqnarray}
where first inequality follows from
\eqref{eqn:HeatSemigroupEstimate-1}, second one follows from
H\"older inequality and the last inequality follows from estimate
\eqref{eqn:LocalH1pBound}.

Combining \eqref{eqn:term1Est}, \eqref{eqn:term2Est} and
\eqref{eqn:term3Est} we get \eqref{eqn:UniformContinuity-1}.
\end{proof}
\begin{remark}
Authors believe that the Lemma \ref{lem:LocalBlowUpBehaviour}
holds also for the critical case of $p=d$. It would be interesting
to prove this fact.
\end{remark}
%\begin{remark}
%The results of this section can be easily extended to Burgers
%equation on torus.
%\end{remark}

\section{Global existence of solution on the torus $\mathbb T^d$}
In this section we establish main results of the article. First,
we will show that there exist global solution of Burgers equation
on torus.
\begin{theorem}\label{thm:GlobalExistenceTorus}
Fix $p>d$. Let $\theta\in (0,1)$, $u_0\in L^p(\mathbb T^d,\Rd)$,
$f\in L^{\infty}([0,T],L^{\frac{2p}{3}}(\mathbb T^d,\Rd)\cap
L^p(\mathbb T^d,\Rd))$, $f\in L^1([0,T];L^{\infty}(\mathbb
T^d,\Rd))$, $f\in C^{\theta}([\eps,T],L^p(\mathbb T^d,\Rd))$,
$\forall \eps>0$. Then there exist global solution $u\in
C([0,T],L^p(\mathbb T^d,\Rd))\cap C^1((0,T];L^p(\mathbb
T^d,\Rd))\cap C((0,T];H^{2,p}(\mathbb T^d,\Rd))\cap
C^{\theta}([\eps,T],H^{2,p}(\mathbb T^d,\Rd))\cap
C^{1+\theta}([\eps,T],L^p(\mathbb T^d,\Rd))$, $\forall \eps>0$.
which satisfies Burgers equation \eqref{eqn:BurgersEquation-11}.
\end{theorem}
\begin{proof}[Proof of Proposition \ref{thm:GlobalExistenceTorus}]
We have according to the Corollary \ref{cor:ClassicalSolutionReg}
that there exist local solution on interval $[0,T_{max})$.
Furthermore, we have by Lemma \ref{lem:LocalBlowUpBehaviour}
blow-up of the solution when $t\to T_{max}$. Thus it is enough to
prove  $L^p$ estimate uniform on semiinterval $[T_0,T_{max})$ for
some $T_0<T_{max}$. Fix $0<\delta<T<T_{max}$. By Corollary
\ref{cor:ClassicalSolutionReg} we can assume that $u\in
C([\eps,T],H^{2,2p}(\mathbb T^d,\Rd))\cap
C^1([\eps,T],L^{2p}(\mathbb T^d,\Rd))$ $\forall \eps>0$. Define
flow
\begin{eqnarray}
dX_t(x) &=& -u(T-t,X_t(x))dt+\sqrt{2\nu}dW_t\nonumber\\
X_0(x) &=& x, x\in \mathbb T^d,0\leq t\leq
T-\delta\label{eqn:FlowDef-1}
\end{eqnarray}
Notice that $u\in C([\delta,T],H^{2,p}(\mathbb T^d,\Rd))\subset
C([\delta,T],C_b^{2-d/p}(\mathbb T^d,\Rd))$ and, therefore, the
flow is correctly defined and does not blow up. Now we will deduce
Feynman-Kac type representation for solution of Burgers equation.
Let $\{u_{\eps}\}_{\eps>0}\in C^1([\delta,T],C^2(\mathbb
T^d,\Rd))$ be a sequence of functions converging to $u$ in
$C^1([\delta,T],L^{2p}(\mathbb T^d,\Rd))\cap
C([\delta,T],H^{2,2p}(\mathbb T^d,\Rd))$. Such sequence can be
constructed, for example, by mollifying of $u$. Then we have by
Ito formula that
\begin{eqnarray}
u_{\eps}(T-t,X_t(x))=u_{\eps}(T,x)+\nonumber\\
\intl_0^t(\nu\triangle u_{\eps}(T-s,X_s)-(u\nabla)u_{\eps}(T-s,X_s)-\frac{\partial u_{\eps}}{\partial t}(T-s,X_s))ds\nonumber\\
+\sqrt{2\nu}\intl_0^t\frac{\partial u_{\eps}}{\partial
x_j}(T-s,X_s)dW_s, t\in [0,T-\delta].\label{eqn:FeynKacApprox-1}
\end{eqnarray}
The last term is a martingale because $\nabla u_{\eps}\in
C([\delta,T],H^{1,p}(\mathbb T^d,\Rd))\subset C([\delta,T]\times
\mathbb T^d,\Rd)$, $p>d$ by Sobolev embedding theorem. Hence
applying mathematical expectation to equality
\eqref{eqn:FeynKacApprox-1} we get
\begin{eqnarray}
u_{\eps}(T,x)=\mathbb{E}u_{\eps}(T-t,X_t(x))+\nonumber\\
\intl_0^t\mathbb{E}((u\nabla)u_{\eps}+\frac{\partial
u_{\eps}}{\partial t}-\nu\triangle u_{\eps}(T-s,X_s))ds,t\in
[0,T-\delta]\label{eqn:FeynKacApprox-2}
\end{eqnarray}
Now let us show convergence  (w.r.t. norm of $L^{\infty}(\mathbb
T^d,\Rd)$) of all terms in \eqref{eqn:FeynKacApprox-2} when we
tend $\eps$ to $0$. We have
\begin{equation}
\supl_{\mathbb T^d}|u_{\eps}(T,x)-u(T,x)|\leq
|u_{\eps}(T)-u(T)|_{H^{1,p}(\mathbb T^d,\Rd)}\to 0,\eps\to
0,\label{eqn:FeynKacApproxConv-1}
\end{equation}
by definition of $u_{\eps}$. Fix $t\in (0,T-\delta]$. Similarly,
\begin{eqnarray}
|\mathbb{E}u_{\eps}(T-t,X_t(\cdot))-\mathbb{E}u(T-t,X_t(\cdot))|_{L^{\infty}(\mathbb T^d,\Rd)}\leq\nonumber\\
|u_{\eps}(T-t)-u(T-t)|_{L^{\infty}(\mathbb
T^d,\Rd)}\stackrel{\eps\to 0}{\rightarrow}
0.\label{eqn:FeynKacApproxConv-2}
\end{eqnarray}
\begin{eqnarray}
|\intl_0^t\mathbb{E}(u\nabla)[u_{\eps}-u](T-s,X_s)ds|_{L^{\infty}(\mathbb T^d,\Rd)}\leq\nonumber\\
\supl_{t\in[\delta,T]}|u(t)|_{L^{\infty}(\mathbb
T^d,\Rd)}|\intl_0^t\mathbb{E}[|\nabla(u_{\eps}-u)(T-s,X_s)|]ds|_{L^{\infty}(\mathbb
T^d,\Rd)}=(I)\nonumber
\end{eqnarray}
Denote
$$
M_t^{\nu}=e^{-\intl_0^tu(T-s,X_s)dX_s-\nu\intl_0^t|u|^2(T-s,X_s)ds},
t\in [0,T-\delta],\nu>0
$$
$M_t^{\nu}$ is a continuous martingale. Indeed, $u$ is bounded
continuous function and the result follows from Theorem 5.3, p.142
in \cite{[Ikeda-1981]}. We can  notice that
\begin{equation}
\mathbb{E}\left(M_t^{\nu}\right)^2\leq
e^{\nu(T-\delta)\supl_{t\in[\delta,T]}|u(t)|_{L^{\infty}(\mathbb
T^d,\Rd)}}=K<\infty
\end{equation}
Notice that by Girsanov type Theorem (see \cite{[Ikeda-1981]}, p.
180-181) we have that
$$
Eg(X_t(x))=EM_t^{\nu}g(x+\sqrt{2\nu}W_t),g\in L^p(\mathbb
T^d,\mathbb{R}).
$$
Thus we have
$$
\mathbb{E}|\nabla(u_{\eps}-u)(T-s,X_s)|=\mathbb{E}M_t^{\nu}|\nabla(u_{\eps}-u)(T-s,x+\sqrt{2\nu}W_s)|
$$
and
\begin{eqnarray*}
(I)\leq \supl_{t\in[\delta,T]}|u(t)|_{L^{\infty}(\mathbb
T^d,\Rd)}\sqrt{T}|(\intl_0^t(\mathbb{E}|\nabla(u_{\eps}-u)(T-s,X_s)|)^2ds)^{1/2}|_{L^{\infty}(\mathbb
T^d,\Rd)}
\end{eqnarray*}
\begin{eqnarray*}
\leq\supl_{t\in[\delta,T]}|u(t)|_{L^{\infty}(\mathbb
T^d,\Rd)}\sqrt{T}|(\intl_0^t\mathbb{E}(M_s^{\nu})^2\mathbb{E}|\nabla(u_{\eps}-u)(T-s,x+\sqrt{2\nu}W_s)|^2ds)^{1/2}|_{L^{\infty}(\mathbb
T^d,\Rd)}
\end{eqnarray*}
\begin{eqnarray*}
\leq\supl_{t\in[\delta,T]}|u(t)|_{L^{\infty}(\mathbb
T^d,\Rd)}\sqrt{KT}(\intl_0^t|\mathbb{E}|\nabla(u_{\eps}-u)(T-s,x+\sqrt{2\nu}W_s)|^2|_{L^{\infty}(\mathbb
T^d,\Rd)}ds)^{1/2}
\end{eqnarray*}
\begin{eqnarray*}
\leq\supl_{t\in[\delta,T]}|u(t)|_{L^{\infty}(\mathbb
T^d,\Rd)}\sqrt{KT}(\intl_0^t|S_s^{\nu}[|\nabla(u_{\eps}-u)(T-s,x)|^2]|_{L^{\infty}(\mathbb
T^d,\Rd)}ds)^{1/2}
\end{eqnarray*}
\begin{eqnarray*}
\leq\supl_{t\in[\delta,T]}|u(t)|_{L^{\infty}(\mathbb
T^d,\Rd)}\sqrt{KT}(\intl_0^t|S_s^{\nu}[|\nabla(u_{\eps}-u)(T-s,x)|^2]|_{H^{1,p}(\mathbb
T^d,\Rd)}ds)^{1/2}
\end{eqnarray*}
\begin{eqnarray*}
\leq\supl_{t\in[\delta,T]}|u(t)|_{L^{\infty}(\mathbb
T^d,\Rd)}\sqrt{KT}(\intl_0^t\frac{1}{s^{1/2}}|[|\nabla(u_{\eps}-u)(T-s,x)|^2]|_{L^p(\mathbb
T^d,\Rd)}ds)^{1/2}
\end{eqnarray*}
\begin{eqnarray*}
\leq\supl_{t\in[\delta,T]}|u(t)|_{L^{\infty}(\mathbb
T^d,\Rd)}\sqrt{KT}(\intl_0^t\frac{|u_{\eps}-u(T-s,x)|_{H^{1,2p}(\mathbb
T^d,\Rd)}^2}{s^{1/2}}ds)^{1/2}
\end{eqnarray*}
\begin{eqnarray*}
\leq\supl_{t\in[\delta,T]}|u(t)|_{L^{\infty}(\mathbb
T^d,\Rd)}\sqrt{K}T^{\frac{3}{4}}\supl_{s\in[\delta,T]}|u_{\eps}-u(s,x)|_{H^{1,2p(\mathbb
T^d,\Rd)}}\stackrel{\eps\to 0}{\rightarrow}0,t\in (0,T-\delta].
\end{eqnarray*}
Thus we have shown convergence of
$\intl_0^t\mathbb{E}(u\nabla)u_{\eps}(T-s,X_s)ds$ to
$\intl_0^t\mathbb{E}(u\nabla)u(T-s,X_s)ds$ in $L^{\infty}(\mathbb
T^d,\Rd)$-norm. Similarly, we have
\begin{eqnarray*}
|\intl_0^t\mathbb{E}(u_{\eps}'-u')(T-s,X_s)ds|_{L^{\infty}(\mathbb T^d,\Rd)}\\
=|\intl_0^t\mathbb{E}M_s^{\nu}
(u_{\eps}'-u')(T-s,x+\sqrt{2\nu}W_s)ds|_{L^{\infty}(\mathbb
T^d,\Rd)}
\end{eqnarray*}
\begin{eqnarray*}
\leq |\sqrt{T}(\intl_0^t(\mathbb{E}[M_s^{\nu}(u_{\eps}'-u')
(T-s,x+\sqrt{2\nu}W_s)])^2ds)^{1/2}|_{L^{\infty(\mathbb
T^d,\Rd)}}\\\leq \sqrt{T}|\intl_0^t
\mathbb{E}(M_s^{\nu})^2\mathbb{E}[| (u_{\eps}'-u')
(T-s,x+\sqrt{2\nu}W_s)|^2]|_{L^{\infty}(\mathbb T^d,\Rd)}^{1/2}
\end{eqnarray*}
\begin{eqnarray*}\leq
\sqrt{TK}|\intl_0^t
S_s^{\nu}[|(u_{\eps}'-u')(T-s,x)|^2]ds|_{L^{\infty}(\mathbb
T^d,\Rd)}^{1/2}\\\leq \sqrt{TK}(\intl_0^t
|S_s^{\nu}[|(u_{\eps}'-u')(T-s,x)|^2]|_{H^{1,p}(\mathbb
T^d,\Rd)}ds)^{1/2}
\end{eqnarray*}
\begin{eqnarray*} \leq
\sqrt{TK}(\intl_0^t\frac{||u_{\eps}'-u'|^2(T-s,x)|_{L^p(\mathbb T^d,\Rd)}}{s^{1/2}}ds)^{1/2}\\
\leq
\sqrt{K}T^{3/4}\supl_{s\in[\delta,T]}|u_{\eps}'(s)-u'(s)|_{L^{2p}(\mathbb
T^d,\Rd)}\to 0,\eps\to 0,t\in (0,T-\delta].
\end{eqnarray*}
For the last  term we have an estimate
\begin{eqnarray*}
|\intl_0^t\mathbb{E}(\triangle u_{\eps}-\triangle u)(T-s,X_s)ds|_{L^{\infty}(\mathbb T^d,\Rd)}\\
=|\intl_0^t\mathbb{E}M_s^{\nu} (\triangle u_{\eps}-\triangle
u)(T-s,x+\sqrt{2\nu}W_s)ds|_{L^{\infty}(\mathbb T^d,\Rd)}\\\leq
|\sqrt{T}(\intl_0^t(\mathbb{E}[M_s^{\nu}(\triangle
u_{\eps}-\triangle u)
(T-s,x+\sqrt{2\nu}W_s)])^2ds)^{1/2}|_{L^{\infty}(\mathbb
T^d,\Rd)}\\\leq \sqrt{T}|\intl_0^t
\mathbb{E}(M_s^{\nu})^2\mathbb{E}[| \triangle (u_{\eps}-u)
(T-s,x+\sqrt{2\nu}W_s)|^2]|_{L^{\infty}(\mathbb
T^d,\Rd)}^{1/2}\\\leq \sqrt{TK}|\intl_0^t S_s^{\nu}[|\triangle
(u_{\eps}-u)(T-s,x)|^2]ds|_{L^{\infty}}^{1/2}\\\leq
\sqrt{TK}(\intl_0^t
|S_s^{\nu}[|\triangle (u_{\eps}-u)(T-s,x)|^2]|_{H^{1,p}(\mathbb T^d,\Rd)}ds)^{1/2}\\
\leq
\sqrt{TK}(\intl_0^t\frac{||\triangle (u_{\eps}-u)|^2(T-s,x)|_{L^p(\mathbb T^d,\Rd)}}{s^{1/2}}ds)^{1/2}\\
\leq
\sqrt{K}T^{3/4}\supl_{s\in[\delta,T]}|u_{\eps}(s)-u(s)|_{H^{2,2p}(\mathbb
T^d,\Rd)}\to 0,\eps\to 0,t\in (0,T-\delta].
\end{eqnarray*}
Thus, we have shown that we can tend $\eps\to 0$ in equality
\eqref{eqn:FeynKacApprox-2}. As a result we get
\begin{eqnarray}
u(T,x)=\mathbb{E}u(T-t,X_t(x))+\nonumber\\
\intl_0^t\mathbb{E}((u\nabla)u+\frac{\partial u}{\partial
t}-\nu\triangle u(T-s,X_s))ds,t\in
[0,T-\delta].\label{eqn:FeynKacApprox-3}
\end{eqnarray}
Put $t=T-\delta$ in equality \eqref{eqn:FeynKacApprox-3}. We have
\begin{eqnarray}
u(T,x)=\mathbb{E}u(\delta,X_t(x))+\nonumber\\
\intl_0^{T-\delta}\mathbb{E}f(T-s,X_s)ds.\label{eqn:FeynKacApprox-4}
\end{eqnarray}
As a consequence we immediately get
\begin{equation}
|u(T)|_{L^{\infty}}\leq
|u(\delta)|_{L^{\infty}}+\intl_0^T|f(s)|_{L^{\infty}(\mathbb
T^d,\Rd)}ds.\label{eqn:FeynKacEst-1}
\end{equation}
Therefore, because torus is compact we have $L^{\infty}(\mathbb
T^d,\Rd)\subset L^p(\mathbb T^d,\Rd)$ and
\begin{equation}
|u(T)|_{L^{p}(\mathbb T^d,\Rd)}\leq
C(|u(\delta)|_{L^{\infty}(\mathbb
T^d,\Rd)}+\intl_0^T|f(s)|_{L^{\infty}(\mathbb
T^d,\Rd)}ds).\label{eqn:FeynKacEst-2}
\end{equation}
Since $u\in C((0,T],H^{1,p}(\mathbb T^d,\Rd))$ and $\delta>0$ is
arbitrary small we have $|u(\delta)|_{L^{\infty}(\mathbb
T^d,\Rd)}\leq |u(\delta)|_{H^{1,p}(\mathbb T^d,\Rd)}<\infty$.
Tending $T\to T_{max}$ in \eqref{eqn:FeynKacEst-2} we get our
estimate.
\end{proof}
The case of Burgers equation in Euclidean space is much more
difficult because $L^{\infty}$ estimate does not allow us to
deduce estimate in $L^p$. In this case we have only following
"conditional" Theorem.
\begin{theorem}\label{thm:GlobalExistence}
Fix $p\in (d,\infty)$. Assume that $u\in
L^{\infty}(0,T;L^p(\Rd,\Rd))$,$\forall T<T_0$ ($T_0$ is such that
$\limsupl_{t\nearrow T_0}|u(t)|_{L^p}^2=\infty$) local solution of
Burgers equation such that $u\in
C^{1,2}((0,T]\times\Rd)$,$u(0)=u_0\in L^p(\Rd,\Rd)$,$f\in
L^p(0,T;L^p(\Rd,\Rd))\cap C^{0,1}((0,T]\times\Rd)$, $\diver f\in
L^{\infty}(0,T_0;L^{\infty}(\Rd))$. Assume also that
\begin{eqnarray}
\omega=\curl u\in
L^{\infty}(0,T_0;L^{\infty}(\Rd)),\label{eqn:BoundedVorticityCond}
\end{eqnarray}
and for any $\delta>0$ there exists $0\leq t_{\delta}<\delta$ such
that $\diver u$ satisfies following growth condition:
\begin{eqnarray}
\exists c>0\,\liminf_{R\to\infty}
e^{-cR^2}[\maxl_{|x|=R,t\in[t_{\delta},T]}\diver u(x,t)]\leq
0,\forall T<T_0,\label{eqn:Phragmen-LindelofCond} \\
\exists 0<t_0<T\supl_x\diver u(t_0,x)\leq
M<\infty.\label{eqn:BoundedConvergenceCond}
\end{eqnarray}
Furthermore, we assume that $u$ has no more than linear growth at
infinity:
\begin{equation}
\limsup_{R\to\infty}\frac{\maxl_{|x|=R,t\in[t_{\delta},T]}
|u(x,t)|}{R}<\infty\forall T<T_0.\label{eqn:MaximumPrincipleCond}
\end{equation}
%or
%
%\begin{eqnarray}
%\exists c>0\,\liminf_{R\to\infty}
%e^{-cR^2}[\maxl_{|x|=R,t\in[0,T]}\diver u(x,t)]\leq
%0,\forall T<T_0,\label{eqn:Phragmen-LindelofCond'} \\
%\supl_x\diver u(0,x)\leq
%M<\infty.\label{eqn:BoundedConvergenceCond'}
%\end{eqnarray}
Let $K=p+M+|\omega|_{L^{\infty}(0,T_0;L^{\infty}(\Rd))}+|\diver
f|_{L^{\infty}(0,T_0;L^{\infty}(\Rd))}$. Then $T_0=\infty$.
Moreover, if $K\geq 0$ we have
\begin{eqnarray}
|u(t)|_{L^p(\Rd,\Rd)}^p+\nu
p(p-1)\intl_0^t\intl_{\Rd}\suml_i|u^i|^{p-2}(s,x)|\nabla
u^i(s,x)|^2dxds\nonumber\\
\leq
|u_0|_{L^p}^pe^{Kt}+\intl_{0}^t|f(s)|_{L^p}^pe^{K(t-s)}ds,t\in
(0,\infty).\label{eqn:EnergyInequality}
\end{eqnarray}
Furthermore, if $K<0$ we have
\begin{equation}
|u(t)|_{L^p(\Rd,\Rd)}^p\leq
|u_0|_{L^p}^pe^{Kt}+\intl_{0}^t|f(s)|_{L^p}^pe^{K(t-s)}ds,t\in
(0,\infty).\label{eqn:EnergyInequality'}
\end{equation}
\end{theorem}
\begin{remark}
Similar condition for Navier-Stokes equation is called
Beale-Kato-Majda condition (see \cite{BealeKatoMajda}).
\end{remark}
\begin{remark}
In the case when compatibility conditions are satisfied and we
have that $\diver u\in C([0,T]\times\Rd)$ we can put $t_0=0$ in
the condition \eqref{eqn:BoundedConvergenceCond}.
\end{remark}
\begin{remark}
If $K<0$ and $\intl_{0}^t|f(s)|_{L^p}^pe^{K(t-s)}ds\to
0,t\to\infty$ than we immediately get that $u(t)\to 0,t\to\infty$
in $L^p$ norm.
\end{remark}
\begin{proof}[Proof of Theorem \ref{thm:GlobalExistence}]
Assume $T_0<\infty$. Fix $t_0>0$ such that
\begin{equation}
\supl_x\diver u(t_0,x)\leq M+1\label{eqn:BoundedConvergenceCond-2}
\end{equation}
and
\begin{equation}
\liminf_{R\to\infty} e^{-cR^2}[\maxl_{|x|=R,t\in[t_0,T]}\diver
u(x,t)]\leq 0,\forall T<T_0.\label{eqn:Phragmen-LindelofCond-2}
\end{equation}
Existence of such $t_0$ follows from
\eqref{eqn:BoundedConvergenceCond} and
\eqref{eqn:Phragmen-LindelofCond}.% or from
%\eqref{eqn:BoundedConvergenceCond'} and
%\eqref{eqn:Phragmen-LindelofCond'}.
Let us multiply $i$-th equation of system
\eqref{eqn:BurgersEquation-1} on
$\sgn(u^i)|u^i|^{p-1},i=1,\ldots,d$, take a sum w.r.t. $i$ and
integrate  w.r.t. to space variable. We get
\begin{eqnarray}
\frac{d}{dt}|u(t)|_{L^p(\Rd,\Rd)}^p+\nu
p(p-1)\intl_{\Rd}\suml_i|u^i|^{p-2}(s,x)|\nabla
u^i(s,x)|^2dx\nonumber\\
=\intl_{\Rd}\suml_i|u_i(t,x)|^p\diver u
dx+p\intl_{\Rd}\suml_if^i\sgn(u^i)|u^i|^{p-1}dx\label{eqn:aux-1}
\end{eqnarray}
Fix $t_1\geq t_0$. Integrating  w.r.t. to time from $t_1$ to $t$
and applying Young inequality we get
\begin{eqnarray}
|u(t)|_{L^p(\Rd,\Rd)}^p+\nu
p(p-1)\intl_{t_1}^t\intl_{\Rd}\suml_i|u^i|^{p-2}(s,x)|\nabla
u^i(s,x)|^2dxds\nonumber\\
\leq
|u(t_1)|_{L^p(\Rd,\Rd)}^p+p\intl_{t_1}^t\intl_{\Rd}\suml_i|u^i|^{p-1}|f_i|dxds+\nonumber\\
\intl_{t_1}^t\intl_{\Rd}\suml_i|u^i(s,x)|^p\diver u(s,x)dxds\nonumber\\
\leq |u(t_1)|_{L^p(\Rd,\Rd)}^p+\intl_{t_1}^t|f(s)|_{L^p}^pds+(p-1)\intl_{t_1}^t|u(s)|_{L^p}^pds+\nonumber\\
\intl_{t_1}^t\intl_{\Rd}\suml_i|u^i(s,x)|^p\diver
u(s,x)dxds.\label{eqn:EnergyIneqAux-1}
\end{eqnarray}
Now let us denote $r=\diver u$.Taking $\diver$ of equation
\eqref{eqn:BurgersEquation-1} we get
\begin{equation}
\frac{\partial r}{\partial t}+(u\nabla)r-\nu\triangle r+|\nabla
u|^2-|\curl u|^2-\diver f=0\label{eqn:DivergenceEqn-1}
\end{equation}
Indeed
\begin{eqnarray}
\diver (u\nabla)u=(u\nabla)\diver u+\suml_{i,j}\frac{\partial
u^i}{\partial x^j}\frac{\partial u^j}{\partial x^i}\nonumber\\
=(u\nabla)\diver u+\suml_{i,j}(\frac{\partial u^i}{\partial
x^j})^2+\frac{\partial u^i}{\partial x^j}(\frac{\partial
u^j}{\partial x^i}-\frac{\partial u^i}{\partial
x^j})\nonumber\\
=(u\nabla)\diver u+|\nabla u|^2-\suml_{i<j}(\frac{\partial
u^i}{\partial x^j}-\frac{\partial u^j}{\partial
x^i})^2\label{eqn:aux-2}
\end{eqnarray}
Let us denote
$$
D=\{(t,x)\in [t_0,T]\times\Rd|r(t,x)\geq 0\},
$$
$$
D^{+}=\{(t,x)\in [t_0,T]\times\Rd|r(t,x)\geq 0,|\nabla
u|^2(t,x)-|\curl u|^2(t,x)-\diver f\geq 0\},
$$
$$
D^{-}=\{(t,x)\in [t_0,T]\times\Rd|r(t,x)\geq 0,|\nabla
u|^2(t,x)-|\curl u|^2(t,x)-\diver f<0\}.
$$
Then $D=D^{+}\cup D^{-}$ and we have that
\begin{equation}
r(t,x)=\diver u(t,x)\leq|\nabla u|(t,x)< |\curl u|(t,x)+|\diver
f|,(t,x)\in D^{-}.\label{eqn:SuplEstimate-1}
\end{equation}
Furthermore, for all $(t,x)\in D^{+}$ we have that
\begin{equation}
\nu\triangle r-(u\nabla)r-\frac{\partial r}{\partial t}=|\nabla
u|^2-|\curl u|^2-\diver f\geq 0,
\end{equation}
$u$ has no more than linear growth on the set $[t_0,T]\times\Rd$
because $u\in C^{1,2}([t_0,T]\times\Rd)$ and condition
\eqref{eqn:MaximumPrincipleCond} is satisfied. Moreover, condition
\eqref{eqn:Phragmen-LindelofCond-2} is also satisfied. Therefore,
by Phragmen-Lindelof principle (see
\cite{ProtterWeinberger},chapter 3, section 6, theorem 10 and
remark (i) after the proof of thm. 10) we have that
\begin{eqnarray}
r(t,x)\leq \max(\supl_{y\in \Rd}\diver u(t_0,y),\supl_{s\in
(t_0,T),y\in\partial D^{+}}r(s,y))\leq \supl_{y\in \Rd}\diver
u(t_0,y)\nonumber\\
+\supl_{s\in (t_0,T),y\in\partial D^{+}}|\curl u(s,y)|+|\diver
f|\leq M+1\nonumber\\
+|\curl
u|_{L^{\infty}(0,T_0;L^{\infty}(\Rd))}+|\diver f|_{L^{\infty}(0,T_0;L^{\infty}(\Rd))},\label{eqn:SuplEstimate-2}\\
(t,x)\in D^{+}\cap \{t_0<t<T,x\in\Rd\}.\nonumber
\end{eqnarray}
Combining estimates \eqref{eqn:SuplEstimate-1} and
\eqref{eqn:SuplEstimate-2} we get
\begin{eqnarray}
r(t,x)=\diver u(t,x)\leq M+1+|\curl
u|_{L^{\infty}(0,T_0;L^{\infty}(\Rd))}\label{eqn:SuplEstimate-3}\\
+|\diver f|_{L^{\infty}(0,T_0;L^{\infty}(\Rd))}, (t,x)\in D\cap
\{t_0<t<T,x\in\Rd\}.\nonumber
\end{eqnarray}
Thus, combining estimate \eqref{eqn:SuplEstimate-3} and inequality
\eqref{eqn:EnergyIneqAux-1} we get
\begin{eqnarray}
|u(t)|_{L^p(\Rd,\Rd)}^p+\nu
p(p-1)\intl_{t_1}^t\intl_{\Rd}\suml_i|u^i|^{p-2}(s,x)|\nabla
u^i(s,x)|^2dxds\nonumber\\
\leq |u(t_1)|_{L^p(\Rd,\Rd)}^p+\intl_{t_1}^t|f(s)|_{L^p}^pds\label{eqn:EnergyIneqAux-2}\\
+\intl_{t_1}^t(p+M+|\curl
u|_{L^{\infty}(0,T_0;L^{\infty}(\Rd))}+|\diver
f|_{L^{\infty}(0,T_0;L^{\infty}(\Rd))})|u(s)|_{L^p}^pds\nonumber
\end{eqnarray}
Denote $K=p+M+|\omega|_{L^{\infty}(0,T_0;L^{\infty}(\Rd))}+|\diver
f|_{L^{\infty}(0,T_0;L^{\infty}(\Rd))}$. Then we can rewrite
\eqref{eqn:EnergyIneqAux-2} as follows
\begin{eqnarray}
|u(t)|_{L^p(\Rd,\Rd)}^p&-&|u(t_1)|_{L^p(\Rd,\Rd)}^p\nonumber\\
&+&\nu
p(p-1)\intl_{t_1}^t\intl_{\Rd}\suml_i|u^i|^{p-2}(s,x)|\nabla
u^i(s,x)|^2dxds\nonumber\\
&\leq&\intl_{t_1}^t|f(s)|_{L^p}^pds+\intl_{t_1}^tK|u(s)|_{L^p}^pds\label{eqn:EnergyIneqAux-3}
\end{eqnarray}
Dividing \eqref{eqn:EnergyIneqAux-3} on $(t-t_1)$ and tending
$t_1$ to $t$ we get
\begin{eqnarray}
\frac{d}{dt}|u(t)|_{L^p(\Rd,\Rd)}^p&+&\nu
p(p-1)\intl_{\Rd}\suml_i|u^i|^{p-2}(t,x)|\nabla u^i(t,x)|^2dx\nonumber\\
&\leq& |f(t)|_{L^p}^p+K|u(t)|_{L^p}^p,t\in (t_0,T).
\end{eqnarray}
Denote
$$
v(t)=|u(t_0)|_{L^p(\Rd,\Rd)}^pe^{K(t-t_0)}+\intl_{t_0}^t|f(s)|_{L^p}^pe^{K(t-s)}ds,
t\in [t_0,T].
$$
Then $v(t_0)=|u(t_0)|_{L^p(\Rd,\Rd)}^p$ and
\begin{equation}
\frac{d}{dt}v(t)=|f(t)|_{L^p}^p+Kv(t),t\in [t_0,T].
\end{equation}
Consequently, $|u(t_0)|_{L^p(\Rd,\Rd)}^p-v(t_0)=0$ and
\begin{equation}
\frac{d}{dt}(|u(t)|_{L^p(\Rd,\Rd)}^p-v(t))\leq
K(|u(t)|_{L^p(\Rd,\Rd)}^p-v(t)),t\in (t_0,T),
\end{equation}
Therefore, by Gronwall lemma \ref{lem:GronwallDifferentialForm} we
have that
\begin{equation}
|u(t)|_{L^p(\Rd,\Rd)}^p\leq
|u(t_0)|_{L^p(\Rd,\Rd)}^pe^{K(t-t_0)}+\intl_{t_0}^t|f(s)|_{L^p}^pe^{K(t-s)}ds,t\in[t_0,T].\label{eqn:EnergyInequality-2}
\end{equation}
Tending $t_0$ to $0$ we get inequality
\eqref{eqn:EnergyInequality'}. Furthermore, in the case of $K\geq
0$, inserting inequality \eqref{eqn:EnergyInequality-2} in the
right part of inequality \eqref{eqn:EnergyIneqAux-2} we get
\begin{eqnarray}
|u(t)|_{L^p(\Rd,\Rd)}^p+\nu
p(p-1)\intl_{t_0}^t\intl_{\Rd}\suml_i|u^i|^{p-2}(s,x)|\nabla
u^i(s,x)|^2dxds\nonumber\\
\leq
|u(t_0)|_{L^p(\Rd,\Rd)}^pe^{K(t-t_0)}+\intl_{t_0}^t|f(s)|_{L^p}^pe^{K(t-s)}ds,t\in[t_0,T].\label{eqn:EnergyInequality-2'}
\end{eqnarray}
% Consequently, by Gronwall lemma, we get that
%\begin{eqnarray}
%|u(t)|_{L^p(\Rd,\Rd)}^p+\nu
%p(p-1)\intl_{t_0}^t\intl_{\Rd}\suml_i|u^i|^{p-2}(s,x)|\nabla
%u^i(s,x)|^2dxds\nonumber\\
%\leq
%(|u(t_0)|_{L^p}^p+\intl_{t_0}^t|f(s)|_{L^p}^pds)(1+Kt)e^{Kt},t\in
%[t_0,T)\label{eqn:EnergyInequality-2}
%\end{eqnarray}
%where $K=p+M+|\omega|_{L^{\infty}(0,T_0;L^{\infty}(\Rd))}+|\diver
%f|_{L^{\infty}(0,T_0;L^{\infty}(\Rd))}$.
Tending $t_0$ to $0$ we
get inequality \eqref{eqn:EnergyInequality}. Tending $t$ to $T_0$
we get contradiction.
\end{proof}
\begin{corollary}\label{cor:GlobalExistence}
Fix $p>d$. Assume that $u_0\in L^p(\Rd,\Rd)$, $f\in
L^{\infty}([0,T],L^{\frac{2p}{3}}(\Rd,\Rd)\cap L^p(\Rd,\Rd))$,
$f\in C^{\theta}([\eps,T],H^{4,p}(\Rd,\Rd))$, $\forall \eps>0$,
$\diver f\in L^{\infty}(0,T;L^{\infty}(\Rd))$. Then there exists
unique local solution $u\in L^{\infty}(0,T_0;L^p(\Rd,\Rd))\cap
C^{1,2}((0,T_0]\times\Rd)$ for some $T_0<T$. Furthermore, if this
local solution satisfies conditions
\eqref{eqn:BoundedVorticityCond},\eqref{eqn:Phragmen-LindelofCond},
\eqref{eqn:BoundedConvergenceCond} on interval $[0,T_0]$ than it
is global solution i.e. $T_0=T$ and energy type inequality
\eqref{eqn:EnergyInequality} is satisfied (with corresponding
$p$).
\end{corollary}
\begin{proof}[Proof of Corollary \ref{cor:GlobalExistence}] Existence of local solution follows from
Corollary \ref{cor:ClassicalSolutionReg-2}. Local solution
satisfies condition \eqref{eqn:MaximumPrincipleCond} by Sobolev
Embedding Theorem (Proposition 2.4, p.5 of \cite{Taylor}). Now
Proof immediately follows from Lemma
\ref{lem:LocalBlowUpBehaviour} and Theorem
\ref{thm:GlobalExistence}.
\end{proof}
\begin{remark}
It is possible to prove in the same way similar theorem and
corollary for torus. In this case, conditions
\eqref{eqn:Phragmen-LindelofCond},
\eqref{eqn:BoundedConvergenceCond} and
\eqref{eqn:MaximumPrincipleCond} will disappear.
\end{remark}
\begin{remark}
If initial condition $u_0$ and force $f$ are irrotational (i.e.
$\curl u_0=\curl f=0$) than $\curl u(t)=0$ and condition
\eqref{eqn:BoundedVorticityCond} is satisfied.
\end{remark}
\begin{remark}
Let us consider case $d=2$ and assume for simplicity that $\diver
f=0$. Then on the boundary of $D^{+}$ we will have that
$$
|\nabla u|^2(t,x)=|\curl u|^2(t,x), (t,x)\in \partial D^{+}.
$$
Therefore, we can deduce that
$$
2\det \nabla u(t,x)=(\diver u)^2(t,x),(t,x)\in \partial D^{+}.
$$
Similarly, we would get
$$
2\det \nabla u(t,x)> (\diver u)^2(t,x),(t,x)\in D^{-}.
$$
As a result, one can consider instead of the assumption that
vorticity is bounded, assumption that jacobian is bounded. It
would be interesting to understand physical meaning of such
assumption. It would also be interesting to acquire better
understanding of the structure of the boundary $\partial D^{+}$.
\end{remark}
\begin{remark}
If we consider 2D Navier-Stokes equation without force then
equation for pressure has form
$$
\triangle p=-2\det \nabla v,
$$
where $p$ is a pressure, $v$ is a velocity. As a result, we have
that $p$ is a subharmonic (resp. superharmonic) function if $v$
conserves (resp. changes) orientation. It would be of interest to
understand physical consequences of this fact.
\end{remark}
In the next theorem we show the application of the Corollary
\ref{cor:GlobalExistence} to the Kardar-Zhang-Parisi (KZP)
equation. We formulate it for torus to get rid of the assumptions
on behavior of the solution when $|x|\to\infty$.
\begin{theorem}\label{thm:KZPequationExistence}
Fix $p>d$. Let $\psi_0\in H^{1,p}(\mathbb T^d)$, $h\in
L^{\infty}([0,T],H^{1,\frac{2p}{3}}(\mathbb T^d)\cap
H^{1,p}(\mathbb T^d))$, $h\in C^{\theta}([\eps,T],H^{5,p}(\mathbb
T^d))$, $\forall \eps>0$, $\triangle h\in
L^{\infty}(0,T;L^{\infty}(\mathbb T^d))$. Then there exists unique
solution $\psi^{\nu}\in C(0,T;H^{1,p}(\mathbb T^d))\cap
C^{1,2}((0,T]\times \mathbb T^d)$ of equation
\begin{eqnarray}
\frac{\partial\psi^{\nu}}{\partial t}+|\nabla\psi^{\nu}|^2=\nu \triangle\psi^{\nu} +h\label{eqn:KDP-1}\\
\psi^{\nu}(0)=\psi_0,t\in [0,T],\nu>0.\label{eqn:KDP-2}
\end{eqnarray}
Furthermore,
\begin{eqnarray}
|\psi^{\nu}(t)|_{H^{1,p}}^p\leq
|\psi_0|_{H^{1,p}}^pe^{Kt}+\intl_0^t|h(s)|_{H^{1,p}}^pe^{K(t-s)}ds,t>0,\label{eqn:AprioriEstim}
\end{eqnarray}
where $K=K(h,p,\psi_0)$.
\end{theorem}
\begin{proof}[Proof of Theorem \ref{thm:KZPequationExistence}]
Proof immediately follows from Corollary \ref{cor:GlobalExistence}
and the fact that gradient of solution of KZP equation is a
solution of Burgers equation.
\end{proof}
We can notice that estimate \eqref{eqn:AprioriEstim} is uniform
w.r.t. $\nu$. This leads us to the following Corollary.
\begin{corollary}
Fix $p>d$. Let $\psi_0\in H^{1,p}(\mathbb T^d)$, $\nabla h\in
L^1(0,T;L^{\infty}(\mathbb T^d))$, $\triangle h\in
L^{\infty}(0,T;L^{\infty}(\mathbb T^d))$. Then there exists unique
viscosity solution $\psi\in C(0,T;H^{1,p}(\mathbb T^d))$ of
equation
\begin{eqnarray}
\frac{\partial\psi}{\partial t}+|\nabla\psi|^2=h\label{eqn:KDP-3}\\
\psi(0)=\psi_0,t\in [0,T].\label{eqn:KDP-4}
\end{eqnarray}
Furthermore,
\begin{eqnarray}
|\psi(t)|_{H^{1,p}}^p\leq
|\psi_0|_{H^{1,p}}^pe^{Kt}+\intl_0^t|h(s)|_{H^{1,p}}^pe^{K(t-s)}ds,t>0,\label{eqn:AprioriEstim-2}
\end{eqnarray}
where $K=K(h,p,d,\psi_0)$.
\end{corollary}
\begin{remark}
The main point of this corollary is an estimate
\eqref{eqn:AprioriEstim-2}. Existence and uniqueness of viscosity
solutions has been shown in many works (see survey
\cite{CrLiIsh-1992}, books
\cite{FleSon-1993},\cite{BardiCap-Dol1997} and references
therein).
\end{remark}
\begin{proof}
We can find $h^{\nu}\in
L^{\infty}([0,T],H^{1,\frac{2p}{3}}(\mathbb T^d)\cap
H^{1,p}(\mathbb T^d))$, $h^{\nu}\in
C^{\theta}([\eps,T],H^{5,p}(\mathbb T^d))$, $\forall \eps>0$,
$\triangle h^{\nu}\in L^{\infty}(0,T;L^{\infty}(\mathbb T^d))$
such that
$$
\intl_0^T|\nabla h^{\nu}(s)-\nabla h(s)|_{L^{\infty}(\mathbb
T^d)}ds\to 0,\nu\to 0.
$$
Let $\{\psi^{\nu}\}_{\nu>0}\in C(0,T;H^{1,p}(\mathbb T^d))\cap
C^{1,2}((0,T]\times \mathbb T^d)$ be sequence of solutions of the
system \eqref{eqn:KDP-1}-\eqref{eqn:KDP-2} where we use $h^{\nu}$
instead of $h$ in equality \eqref{eqn:KDP-1}. Since
$H^{1,p}(\mathbb T^d)\subset C(\mathbb T^d),p>d$ and estimate
\eqref{eqn:AprioriEstim} we have uniform w.r.t. $\nu$ estimate
\begin{eqnarray}
|\psi^{\nu}|_{C(0,T;C(\mathbb T^d))}^p\leq
K(T,\psi_0,h,d),T>0,p>d.\label{eqn:AprioriEstim-3}
\end{eqnarray}
Then according to Theorem 1.1, p. 175 in \cite{Barles-2006} we
have that there exist uniformly bounded upper continuous
subsolution $\psi^*=\limsupl_{\nu\to 0}^*\psi^{\nu}$ and uniformly
bounded lower continuous supersolution $\psi_*=\liminfl_{\nu\to
0}^*\psi^{\nu}$ of system \eqref{eqn:KDP-3}-\eqref{eqn:KDP-4}.
Therefore, by comparison principle for viscosity solutions of
Hamilton-Jacobi equations (see Theorem 2, p.585 and Remark 3, p.
593 of \cite{CrIshLi-1987}), $\psi^*\leq\psi_*$ and
$\psi=\psi^*=\psi_*$. Thus, $\psi^{\nu}$ locally uniformly
converges to unique viscosity solution $\psi$ of equation
\eqref{eqn:KDP-3}-\eqref{eqn:KDP-4}. Estimate
\eqref{eqn:AprioriEstim} implies that $\psi$ satisfies
\eqref{eqn:AprioriEstim-2}.
\end{proof}
\section{Appendix}
Let $S$ be an interval of the real line of the form $[a,b]$ or
$[a,\infty)$ with $a<b$. Denote $\dot{S}$ interior part of $S$.
\begin{lemma}[Gronwall lemma in differential form]\label{lem:GronwallDifferentialForm}
Let $u,\beta\in C(S)$, $u$ is differentiable in $\dot{S}$ and
$$
u'(t)\leq\beta(t)u(t),t\in \dot{S}.
$$
Then
$$
u(t)\leq u(a)e^{\intl_a^t\beta(s)ds}.
$$
\end{lemma}
\begin{remark}
Notice that there is no assumption that $\beta$ is nonnegative.
\end{remark}
\begin{proof}[Proof of Lemma \ref{lem:GronwallDifferentialForm}]
Let $v(t)=e^{\intl_a^t\beta(s)ds},t\in S$. Then
$$
v'(t)=\beta(t)v(t),t\in S.
$$
Notice that $v(t)>0,t\in S$ and, therefore,
$$
\frac{d}{dt}\frac{u(t)}{v(t)}=\frac{u'v-v'u}{v^2}\leq \frac{\beta
u v-\beta v u}{v^2}=0. t\in\dot{S},
$$
i.e.
$$
\frac{u(t)}{v(t)}\leq \frac{u(a)}{v(a)}=u(a),
$$
and the result follows.
\end{proof}
%\begin{lemma}[Gronwall lemma in integral form]\label{lem:GronwallIntegralForm}
%Assume that $u,\beta\in C(S)$, $\beta\geq 0$, negative part of
%$\al$ is integrable on every closed and bounded subinterval of
%$S$. If
%$$
%u(t)\leq\al(t)+\intl_a^t\beta(s)u(s)ds,t\in S,
%$$
%then
%$$
%u(t)\leq\al(t)+\intl_a^t\beta(s)\al(s)e^{\intl_s^t\beta(r)dr}ds,t\in
%S.
%$$
%Furthermore, if $\al$ is nondecreasing function on $S$ then
%$$
%u(t)\leq\al(t)e^{\intl_a^t\beta(s)ds}.
%$$
%\end{lemma}
%\begin{proof}[Proof of Lemma \ref{lem:GronwallIntegralForm}]
%Proof is omitted.
%\end{proof}
\begin{proof}[Proof of Theorem \ref{thm:AbstractExistence}]
Denote
$$
Q_T=\{u\in C([0,T];X)\cap C((0,T],Y)\cap
C((0,T];Z)||u|_{Q_T}<\infty\}
$$
where
$$
||\cdot||_{Q_T}=||\cdot||_{C([0,T];X)}+\supl_{t\in(0,T]}t^b|u(t)|_Y+\supl_{t\in(0,T]}t^c|u(t)|_Z.
$$
Then $Q_T$ is a complete metric space.

Fix $u_0\in X$, $g\in L^{\infty}(0,T;W)$ and let $\al$, $\beta$
and $T_1>0$ be such that
\begin{eqnarray}
|R_tu_0+\intl_0^tR_{t-s}gds|_X\leq
\al,t\in(0,T].\label{eqn:AppendixAux-1}\\
t^b|R_tu_0+\intl_0^tR_{t-s}gds|_Y\leq
\beta,\label{eqn:AppendixAux-1'}\\
t^c|R_tu_0+\intl_0^tR_{t-s}gds|_Z\leq
\beta,t\in(0,T_1].\label{eqn:AppendixAux-1''}
\end{eqnarray}
Existence of $\al$ satisfying \eqref{eqn:AppendixAux-1} follows
from the fact that $\{R_t\}_{t\geq 0}$ is $C_0$-semigroup in $X$
and following estimate
\begin{equation}
|\intl_0^tR_{t-s}gds|_X\leq
C\intl_0^t\frac{|g(s)|_W}{|t-s|^a}ds\leq
C|g|_{L^{\infty}(0,t;W)}t^{1-a}.\label{eqn:AppendixAux-2}
\end{equation}
Estimate \eqref{eqn:AppendixAux-2} and assumptions b) and c) of
the theorem imply that for any $\beta>0$ inequalities
\eqref{eqn:AppendixAux-1'}, \eqref{eqn:AppendixAux-1''} are true
for all sufficiently small $T_1>0$.

Let
\begin{eqnarray*}
M(\al,\beta,T)=\left\{u\in Q_T||u|_{C([0,T];X)}\leq
2\al,\supl_{t\in(0,T]}t^b|u(t)|_Y\leq 2\beta,\right.\\
\left.\supl_{t\in(0,T]}t^c|u(t)|_Z\leq 2\beta\right\}.
\end{eqnarray*}
Then $M(\al,\beta,T)$ endowed with norm $Q_T$ is also complete
metric space and we will show that if $\beta,T=T_1>0$ are small
enough than the map $\mathcal{F}:u\mapsto
R_tu_0+\intl_0^tR_{t-s}gds+\intl_0^tR_{t-s}G(u(s))ds$ is a
contraction on $M(\al,\beta,T)$.

We have by \eqref{eqn:a1Cond} following inequality
\begin{eqnarray}
|\mathcal{F}(u)-\mathcal{F}(v)|_X(t)\leq
\intl_0^t|R_{t-s}(G(u)-G(v))|_Xds\nonumber\\
\leq C\intl_0^t\frac{|G(u)-G(v)|_{W}(s)}{|t-s|^a}ds\leq C
\intl_0^t\frac{|u(s)|_Y|u-v|_Z+|v(s)|_Z|u-v|_Y}{|t-s|^a}ds\nonumber\\
\leq
C(\intl_0^t\frac{s^b|u(s)|_Ys^c|u-v|_Z}{s^{b+c}|t-s|^a}ds)+\intl_0^t\frac{s^c|v(s)|_Z|u-v|_Y}{s^{b+c}|t-s|^a}ds\nonumber\\
\leq C\beta |u-v|_{Q_t}t^{1-(a+b+c)}.\label{eqn:AppendixAux-3}
\end{eqnarray}
Similarly,
\begin{eqnarray}
t^b|\mathcal{F}(u)-\mathcal{F}(v)|_Y(t)\leq
t^b\intl_0^t|R_{t-s}(G(u)-G(v))|_Yds\nonumber\\
\leq Ct^b\intl_0^t\frac{|R_{(t-s)/2}(G(u)-G(v))|_X}{(t-s)^b}ds\leq
Ct^b\intl_0^t\frac{|(G(u)-G(v))|_W}{(t-s)^{a+b}}ds\nonumber\\
\leq
Ct^b\intl_0^t\frac{|u(s)|_Y|u-v|_Z+|v(s)|_Z|u-v|_Y}{|t-s|^{a+b}}ds\nonumber\\
\leq
Ct^b(\intl_0^t\frac{s^b|u(s)|_Ys^c|u-v|_Z}{s^{b+c}|t-s|^{a+b}}ds)+\intl_0^t\frac{s^c|v(s)|_Z|u-v|_Y}{s^{b+c}|t-s|^{a+b}}ds\nonumber\\
\leq C\beta |u-v|_{Q_t}t^{1-(a+b+c)}.\label{eqn:AppendixAux-4}
\end{eqnarray}
and
\begin{eqnarray}
t^c|\mathcal{F}(u)-\mathcal{F}(v)|_Z(t)\leq C\beta
|u-v|_{Q_t}t^{1-(a+b+c)}.\label{eqn:AppendixAux-5}
\end{eqnarray}
Combining \eqref{eqn:AppendixAux-3}, \eqref{eqn:AppendixAux-4} and
\eqref{eqn:AppendixAux-5} we get that if $\beta<\frac{1}{C}$ then
$\mathcal{F}$ is a contraction on $Q_t$. Furthermore, it follows
from inequalities \eqref{eqn:AppendixAux-2},
\eqref{eqn:AppendixAux-1}, \eqref{eqn:AppendixAux-1'} and
\eqref{eqn:AppendixAux-1''} that $\mathcal{F}$ is a map from
$M(\al,\beta,T)$ to $M(\al,\beta,T)$. Thus there exists a unique
fixed point $u$ of the map $\mathcal{F}:M(\al,\beta,T)\to
M(\al,\beta,T)$. It remains to show that $u$ has designated
asymptotic behavior when $t\to 0$. It can be done in the same way
as in \cite{Weissler-1980}, p.223-224.

%\begin{equation}
%|\mathcal{F}(u)-\mathcal{F}(v)|_{Q_t}\leq C\beta
%|u-v|_{Q_t}.\label{eqn:AppendixAux-6}
%\end{equation}
\end{proof}
\par\medskip\noindent
\textbf{Acknowledgement.} The authors are indebted to the
anonymous referee for helpful comments.

\end{document}